\documentclass[11pt,a4paper]{article}

\usepackage[utf8]{inputenc} 
\usepackage[english]{babel}    
\usepackage{amsmath,amssymb,amsthm}  
\usepackage[T1]{fontenc}      
\usepackage{url}
\usepackage{textcomp}
\usepackage{latexsym}
\usepackage{tikz}
\usepackage{mathrsfs}

\theoremstyle{plain}
\newtheorem{thm}{Theorem}[section]
\newtheorem{prp}[thm]{Proposition}

\newtheorem{lm}[thm]{Lemma}
\newtheorem{rmq}[thm]{Remark}

\theoremstyle{definition}
\newtheorem{df}[thm]{Definition}
\newtheorem{ex}[thm]{Example}

\newcommand\R{\mathbb{R}} 
 
\newcommand\N{\mathbb{N}}

\newcommand{\lS}{\ell_S}

\title{Fixed points of endomorphisms and  relations between metrics\\in preGarside monoids}
\date{November, 2013}
\author{Oussama AJBAL}

\begin{document}

\maketitle

\begin{abstract}
In \cite{RoS}, it is proved that the fixed points submonoid and the periodic points submonoid of a trace monoid endomorphism are always finitely generated. We show that for finitely generated left preGarside monoids, that includs finitely generated preGarside monoids, Garside monoids and Artin monoids, the fixed and periodic points submonoids of any endomorphism are also finitely generated left preGarside monoids under some condition, and in the case of Artin monoids, these submonoids are always Artin monoids too. We also prove algebraically some inequalities, equivalences and non-equivalences between three metrics in finitely generated preGarside monoids, and especially in trace monoids and Garside monoids.
\end{abstract}

\section*{Introduction}
Trace monoids, or equivalently partially commutative monoids,  are monoids of a particular interest and have been widly studied. In particular, they are Artin Monoids. This explains why they are also called Right Angle Artin monoids (RAAM for short). In \cite{RoS, RoS1}, the authors consider endomorphisms of trace monoids. They study the submonoid of fixed points  of such an endomorphism and prove that it is finitely generated. They also caracterise those endomorphisms that are contractions relatively to some natural distance on trace monoids. In~\cite{Cri2}, Crisp obtained similar results for Artin monoids. He proved that the submonoid of fixed points is even finitely presented, but only in the special  case of an isomorphism. Here, we aim to unified and extend both results to all Artin monoids and, more generally,  to the larger classes of preGarside monoids.\\

Given a monoid $M$, we denote by $\mathrm{End}(M)$ the endomorphism monoid of $M$. For 
$\varphi\in\mathrm{End}(M)$, we say that $x\in M$ is a {\em fixed point} of $\varphi$ if 
$\varphi(x)=x$. If $\varphi^n(x)=x$ for some $n\geq1$, we say that $x$ is a {\em periodic 
point} of $\varphi$. Let $\mathrm{Fix}(\varphi)$ (respectively $\mathrm{Per}(\varphi)$) denote 
the submonoid of all fixed points (respectively periodic points) of $\varphi$. Clearly,
\[\mathrm{Per}(\varphi)=\bigcup_{n\geq1}\mathrm{Fix}(\varphi^n).\]

Our First result is the following:

\begin{thm}~
\begin{enumerate}
\item[(i)] If $M$ is a finitely generated left preGarside monoid, with an additive and homogeneous norm $\nu$, and $\varphi$ is in $\mathrm{End}(M)$ such that the morphism $\pi$ is well-defined, then $\mathrm{Fix}(\varphi)$ and $\mathrm{Per}(\varphi)$ are also finitely generated left preGarside monoids.
\item[(ii)] if $M$ is an Artin monoid, and $\varphi$ is in $\mathrm{End}(M)$, then $\mathrm{Fix}(\varphi)$ and $\mathrm{Per}(\varphi)$ are also Artin monoids.
\end{enumerate}
\end{thm}

Similar result  holds if one considers right preGarside monoids.\\

Two distances $d_2$ and $d_3$  on trace monoids have been introduced in~\cite{BMP} and~\cite{Kwi}. They were proved to be uniformaly equivalent in~\cite{KK}. In~\cite{RoS}, the authors caracterise those endomorphisms that are contractions relatively to $d_2$. We will prove that in the general context of Artin monoids, distance $d_2$ and $d_3$ are no more uniformaly equivalent in general. Indeed, distance $d_2$ should be replace by an alternative one, that we denote by $d_1$. We prove that $d_1$ is larger than the other two in the general case, and we will also prove that $d_1$ and $d_3$ are uniformly equivalent in the case of Artin monoids of spherical type.\\

Finally, replacing $d_2$ by $d_1$, we extend \cite[Theorem 4.1]{RoS}, to all Artin monoids  in Theorem \ref{thm4} (we refer to the next sections for notations).

\begin{thm}
Let $M$ be an Artin monoid and $\varphi$ be in  $\mathrm{End}(M)$.
\begin{enumerate}
\item[(i)] $\varphi$ is a contraction with respect to $d_1$;
\item [(ii)] for all $u,v\in M_{\mathrm{red}}$, ~~~$\alpha(uv)=u~~\Rightarrow~~\alpha(\varphi(uv))=\alpha(\varphi(u))$.
\end{enumerate}
\end{thm}

\section{Preliminaries}

We start in this section by defining the monoids we will work with in this paper, namely left preGarside monoids and Artin monoids.\\

Consider a monoid $M$. It is said to be {\em cancellative} if, for all $a,b,c,d\in M$, the equality 
$cad=cbd$ imposes $a=b$. An element $b$ is called a {\em factor} of an element $a$ if we can 
write $a=cbd$ in $M$. We denote by $\mathrm{Div}(a)$ the set of factors of $a$. We denote 
by $\preceq$ left divisibility in $M$ (that is, for $a,b\in M$, we have $a\preceq b$ when there 
exists $c\in M$ such that $b=ac$). Right divisibility (defined similarly : $b\succeq a$ when there exists $c\in M$ such that $b=ca$) will rarely be 
used in this paper, so divisibility in $M$ will simply mean left divisibility. When $a$ is a {\em left divisor} of $b$ in $M$, we say that $b$ is a {\em right multiple} of $a$. An element $a$ is said to 
be {\em balanced} if its sets of right-divisors and of left-divisors are equal, which in this case 
have to be equal to $\mathrm{Div}(a)$.\\

We say that $M$ is {\em atomic} if there exists a mapping $\nu:M\rightarrow\mathbb{N}$, 
called a {\em norm}, satisfying $\nu(a)>0$ for $a\neq1$ and $\nu(ab)\geq\nu(a)+\nu(b)$ for 
all $a,b\in M$. Note that the existence of such a mapping implies that the relations $\preceq$ and 
$\succeq$ are partial orders on $M$. When $\nu(ab)=\nu(a)+\nu(b)$ for all $a,b\in M$, we say that $\nu$ is {\em additive}. An {\em atom} in a monoid is an element $a\in M$ satisfying: $a=bc\Rightarrow b=1$ or $c=1$ for all $b,c\in M$. We denote by $\mathcal S(M)$ the set of atoms of $M$. Note that in an atomic monoid $M$, the set $\mathcal S(M)$ has to be a generating set, and that any generating set of $M$ contains $\mathcal S(M)$. In particular, $M$ is finitely generated if and only if $\mathcal S(M)$ is finite. In an atomic monoid $M$, if $\nu(a)=\nu(b)$ for all $a,b\in\mathcal S(M)$, we say that $\nu$ is {\em homogeneous}. \\

A monoid $M$ is said to be a {\em left preGarside monoid} if
\begin{enumerate}
\item[($a_L$)] it is atomic and left cancellative ($ca=cb~\Rightarrow~a=b$);
\item[($b_L$)] for all $a,b\in M$, if the set $\{c\in M\mid a\preceq c~\text{and}~b\preceq c\}$ is nonempty,
then it has a least element, denoted by $a\vee_Lb$ or $a\vee b$.
\end{enumerate}

It is said to be a {\em right preGarside monoid} if

\begin{enumerate}
\item[($a_R$)] it is atomic and right cancellative ($ac=bc~\Rightarrow~a=b$);
\item[($b_R$)] for all $a,b\in M$, if the set $\{c\in M\mid c\succeq a~\text{and}~c\succeq b\}$ is nonempty,
then it has a least element, denoted by $a\vee_Rb$.
\end{enumerate}

And it is said to be a {\em preGarside monoid} if it is both {\em left preGarside} and {\em right preGarside monoid}.\\

A {\em Garside element} of a preGarside monoid is a balanced element whose set of factors 
generates the whole monoid. When such an element exists, we say that the monoid is a {\em 
Garside monoid}.\\

Given a non empty finite set $S$, a {\em Coxeter matrix} is a symmetric matrix $(m_{ab})_
{a,b\in S}$ with entries in $\{1,2,\cdots,\infty\}$, such that $m_{aa}=1$ and $m_{ab}\geq
2$, for $a\neq b$. A Coxeter system associated to a Coxeter matrix $(m_{ab})_{a,b\in S}$ is a pair 
$(W,S)$, where $W$ is the group with presentation 
\[W=\langle S\mid (ab)^{m_{ab}}=1;~m_{ab}\neq\infty\rangle.\]
The corresponding {\em Artin monoid $M$} is the monoid with presentation
\[M=\langle S\mid[a,b\rangle^{m_{ab}}=[b,a\rangle^{m_{ab}};~m_{ab}\neq\infty\rangle^+.\]
where $[a,b\rangle^m$ denotes the alternating product $aba\cdots$ containing $m$ terms.
Since the defined Artin relations are homogeneous, $M$ has a natural length function $\lS$ compatible 
with the product. In \cite{BrS}, it is shown that every finite subset of $M$ has a {\em greatest common left divisor (gcd)}, and a {\em greatest common right divisor (gcrd)}. It is also shown that a finite subset of $S$ has a {\em least common right multiple (lcm)} if and only if it has a common right multiple, and that in that case, the {\em least common right multiple} and {\em least common left multiple} are equal. For a subset $T\subseteq S$, we denote its {\em lcm} by $\Delta(T)$ when it exists.\\

\section{The submonoids of fixed points and of periodic points}

In general, a fixed points submonoid or any submonoid of a finitely generated monoid is not necessarily finitely generated. 

\begin{ex}\label{ex1}
Consider the cancellative monoid $M=\langle a,e,b\mid ae=ea,ebe=b\rangle^+$ and the endomorphism $\varphi$ such that $\varphi(a)=ae$, $\varphi(e)=e$ and $\varphi(b)=b$. Clearly, $a^nba^n\in\mathrm{Fix}(\varphi)$ for every $n\in \N$. And one can show that $a^nba^n$ is not decomposable in $\mathrm{Fix}(\varphi)$, which means that $\mathrm{Fix}(\varphi)$ is not finitely generated.
\end{ex}

In this section, we will show that for a finitely generated left preGarside monoid $M$, and an endomorphism $\varphi\in\mathrm{End}(M)$, the submonoids $\mathrm{Fix}(\varphi)$ and $\mathrm{Per}(\varphi)$ are also finitely generated left preGarside monoids under some condition (Theorems \ref{thm1} and \ref{thm2}). And that, in the particular case of Artin monoids, $\mathrm{Fix}(\varphi)$ and $\mathrm{Per}(\varphi)$ are not only finitely generated, but even finitely presented (Proposition \ref{prp2}).\\

For a monoid $M$, finitely generated by $S$, and an endomorphism $\varphi\in\mathrm{End}(M)$, we define \[ n_\varphi = \left \{
\begin{array}{ll}
     \mathrm{max}\{k\in\N^*\mid\exists s\in S~\mathrm{such~that}~\varphi^k(s)=1~\mathrm{and}~\varphi^{k-1}(s)\neq1\} & \text{if $1\in\varphi(S)$}\\
     1 & \text{if $1\notin\varphi(S)$}\\
\end{array}.
\right.\]

Given a non empty set $X$, we denote by $X^*$ the set of all finite words $x_1\cdots x_n$ over the elements of $X$, that we call {\em letters}, and by $\varepsilon$ the {\em empty word} in $X^*$. Assume $M$ is a monoid generated by a set~$S$, and let $T\subseteq S$. Let us denote by $\pi_T^*: S^*\to T^*$ the \emph{forgetting} morphism of monoids defined by $\pi_T^*(t) = t$ for $t\in T$, and by $\pi_T^*(t)=\varepsilon$ for $t\in S\setminus T$. In the sequel, when it is well-defined, we denote by $\pi_T:M \rightarrow M$ the morphism of monoids induced by $\pi_T^*$.

\begin{ex}\label{ex2}
For $M=\langle s,t\mid sts=tst\rangle^+$, with $S=\{s,t\}$ and $T=\{t\}$, we have $\pi_T^*(sts)=t$ and $\pi_T^*(tst)=tt$. Or in $M$, one has $t\neq t^2$, then $\pi_T$ is not well-defined. But for $M=\langle s,t\mid stst=tsts\rangle^+$ with similar $S$ and $T$, the morphism $\pi_T$ is well-defined, since $\pi_T(stst)=\pi_T(tsts)=t^2$.
\end{ex}~

\subsection{The submonoid of fixed points}

For this subsection, let $M$ be a finitely generated left preGarside monoid, equiped with an additive and homogeneous norm $\nu$, and $\varphi$ be in $\mathrm{End}(M)$. Set $$S=\mathcal S(M),~~S_0=S\cap\mathrm{Per}(\varphi),~~S_1=S\cap(\varphi^{n_\varphi})^{-1}\{1\},~~S_2=S\setminus S_1,$$ $p=|S|!$, and $\pi:=\pi_{S_2}$ when it is well-defined, which is always the case when $1\notin\varphi(S)$.\\

\begin{thm}\label{thm1}
Let $M$ be a finitely generated left preGarside monoid, with an additive and homogeneous norm $\nu$, $\varphi$ be in $\mathrm{End}(M)$, and $\pi$ as above. If the morphism $\pi$ is well-defined, then $\mathrm{Fix}(\varphi)$ is a finitely generated left preGarside monoid.
\end{thm}~

The proof of this theorem is in the spirit of \cite[Theorem 3.1]{RoS}, where the particular case of trace monoids was considered. We will start though by proving some lemmas. Note that in the proof of the next lemma, where $\varphi|_S$ is a permutation (i.e. $\varphi$ is an automorphism), we do not need the additivity or the homogeneity of the norm $\nu$. In the case of preGarside monoids, this result is shown in \cite[Proposition 2.26]{BDM}.\\

\begin{lm}\label{lm1}
If the restriction of $\varphi$ to $S$ is a permutation, then $\mathrm{Fix}(\varphi)$ is a finitely generated left preGarside monoid.
\end{lm}

\begin{proof}
Since $\mathrm{Fix}(\varphi)\subseteq M$, it is clear that property ($a_L$) holds for $\mathrm{Fix}(\varphi)$. Let $a,b\in\mathrm{Fix}(\varphi)$ be such that the set $\{c\in\mathrm{Fix}(\varphi)\mid a\preceq c~\text{and}~b\preceq c\}$ is nonempty, and set $\delta=a\vee b$ their (left) lcm in $M$. We have $\varphi(a)=a\preceq\varphi(\delta)$ and $\varphi(b)=b\preceq\varphi(\delta)$, then $\delta\preceq\varphi(\delta)$. Thus $\delta\preceq\varphi(\delta)\preceq\cdots\preceq\varphi^p(\delta)$, and so $\nu(\delta)\leq\nu(\varphi(\delta))\leq\cdots\leq\nu(\varphi^p(\delta))$. Or $\varphi|_S$ is a permutation, and $p=|S|!$, then $\varphi^p=\mathrm{Id}_M$. Therefore $\nu(\delta)=\nu(\varphi(\delta))$, with $\delta\preceq\varphi(\delta)$. Thus, we have $\delta=\varphi(\delta)$, and so $\delta\in\mathrm{Fix}(\varphi)$. Let $c\in\mathrm{Fix}(\varphi)$ such that $a\preceq c$ and $b\preceq c$. Write $c=\delta\delta'$ with $\delta'\in M$. We have $c,\delta\in\mathrm{Fix}(\varphi)$, then, by cancellativity, $\delta'\in\mathrm{Fix}(\varphi)$. Hence property ($b_L$) holds, and so $\mathrm{Fix}(\varphi)$ is a left preGarside monoid.\\

Let $\Sigma$ be the set of all $\varphi$-orbits $B$ in $S$ that have a right common multiple (and
therefore a least right common multiple $\Delta(B)$ (\cite[Lemma 2.1]{GoP})). We claim that
\begin{equation}\label{eq1}
\mathrm{Fix}(\varphi)=\langle\Delta(B), B\in\Sigma\rangle^+.
\end{equation}

Let $B\in\Sigma$, and $b\in B$. Since $B$ is a $\varphi$-orbit, then there exists $a\in B$ such that $b=\varphi(a)$. But $a\preceq\Delta(B)$, so $b=\varphi(a)\preceq\varphi(\Delta(B))$. Thus $b\preceq\varphi(\Delta(B))$ for every $b\in B$, and then $\Delta(B)\preceq\varphi(\Delta(B))$. Therefore, as above, we have $\Delta(B)=\varphi(\Delta(B))$. Hence $\langle\Delta(B),B\in\Sigma\rangle^+\subseteq\mathrm{Fix}(\varphi)$.\\
Conversely, let $u\in\mathrm{Fix}(\varphi)$ such that $u\neq1$, and let $s\in S$ such that $s\preceq u$; then $u$ is left divisible by all the elements in the $\varphi$-orbit $B_1$ of $s$, so is left divisible by their lcm $u_1:=\Delta(B_1)$ (which exists). Let $v\in M$ such that $u=u_1v$. Since $u,u_1\in\mathrm{Fix}(\varphi)$, we have by cancellativity, $v\in\mathrm{Fix}(\varphi)$. By induction on $\nu(u)$, we get that $v\in\langle\Delta(B)\mid B\in\Sigma\rangle^+$. Thus $u\in\langle\Delta(B)\mid B\in\Sigma\rangle^+$ and so (\ref{eq1}) holds. The set $S$ is finite, then so is $\Sigma$. Therefore $\mathrm{Fix}(\varphi)$ is finitely generated.
\end{proof}~~

\begin{ex}~\label{ex3}
\begin{enumerate}
\item Let $M=\langle s,t,u\mid st=ts,~sus=usu,~tut=utu\rangle^+$, and $\varphi\in\mathrm{End}(M)$ such that $\varphi(u)=u$, $\varphi(t)=s$ and $\varphi(s)=t$. The monoid $M$ is an Artin monoid, so it satisfies all properties of preGarside monoids by \cite{BrS}. Then we have $\Sigma=\{\{u\},\{s,t\}\}$, and $\mathrm{Fix}(\varphi)=\langle u,st\rangle^+$.
\item Let $M=\langle a_1,b_1,a_2,b_2\mid a_1b_1a_1=b_1^2,~a_2b_2a_2=b_2^2,~a_1a_2=a_2a_1,~b_1b_2=b_2b_1, ~a_1b_2=b_2a_1,~b_1a_2=a_2b_1\rangle^+$, and $\varphi\in\mathrm{End}(M)$ such that $\varphi(a_1)=a_2$, $\varphi(a_2)=a_1$, $\varphi(b_1)=b_2$ and $\varphi(b_2)=b_1$. The monoid $M$ is a preGarside monoid, as a direct product $M=\langle a_1,b_1\mid a_1b_1a_1=b_1^2\rangle^+\times\langle a_2,b_2\mid a_2b_2a_2=b_2^2\rangle^+$ of two preGarside monoids (see \cite{DeP}). Then we have $\Sigma=\{\{a_1,a_2\},\{b_1,b_2\}\}$, and $\mathrm{Fix}(\varphi)=\langle a_1a_2,b_1b_2\rangle^+$.
\end{enumerate}
\end{ex}~

\begin{lm}\label{lm2}
If $1\notin\varphi(S)$, then
\begin{equation}\label{eq2}
S_0=S\cap\mathrm{Fix}(\varphi^p)=S\cap\varphi^p(S),
\end{equation} and 
\begin{equation}\label{eq3}
\varphi(S_0)=S_0.
\end{equation}
\end{lm}

\begin{proof}
Recall that $S_0=S\cap\mathrm{Per}(\varphi)$. Let $a\in S_0$ and $m=\mathrm{min}\{k\in\mathbb{N}^*\slash\varphi^k(a)=a\}$. Since $1\notin \varphi(S)$ and $\nu$ is additive and homogeneous, we have $\nu(\varphi(u))\geq\nu(u)$ for every $u\in M$. Hence $\varphi^n(a)\in S$ for every $n\in\mathbb{N}$. Assume that there exist $0<i<j\leq m$ such that $\varphi^i(a)=\varphi^j(a)$. So $\varphi^{m-j}\circ\varphi^i(a)=\varphi^m(a)$. Then $\varphi^{m-(j-i)}(a)=a$ and $m-(j-i)<m$, which contradicts the definition of $m$. Thus $m=\#\{a,\varphi(a),\cdots,\varphi^{m-1}(a)\}\leq|S|$, and then $m$ divides $p$. Hence $\varphi^p(a)=a$. So $S_0\subseteq S\cap\mathrm{Fix}(\varphi^p)$, and we clearly have $S\cap\mathrm{Fix}(\varphi^p)\subseteq S\cap\varphi^p(S)$.

Let $a\in S\cap\varphi^p(S)$ and $b\in S$ such that $\varphi^p(b)=a$. Since $1\notin\varphi(S)$, using that $\nu$ is additive and homogeneous, and $\varphi^p(b)=a$, we get $\{b,\varphi(b),\cdots,\varphi^p(b)\}\subseteq S$. The inequality $|S|<p+1$ yields $\varphi^i(b)=\varphi^j(b)$ for some $0\leq i<j\leq p$. By composing with $\varphi^{p-i}$, we get $a=\varphi^{j-i}(a)$. Then $a\in\mathrm{Per}(\varphi)\cap S=S_0$, and therefore (\ref{eq2}) holds.\\

Let $a\in S_0$. As before, we have $\{a,\varphi(a),\dots,\varphi^{p-1}(a)\}\subseteq S_0=S\cap\mathrm{Fix}(\varphi^p)$. On the one hand, we have $\varphi(a)\in S_0$, then $\varphi(S_0)\subseteq S_0$. On the other hand, we have $a=\varphi(\varphi^{p-1}(a))$ and $\varphi^{p-1}(a)\in S_0$, then $\varphi(S_0)\supseteq S_0$, and so (\ref{eq3}) holds.
\end{proof}~

Let $M_0$, $M_1$ and $M_2$, be the submonoids of $M$ generated by $S_0$, $S_1$ and $S_2$ respectively. By definition of $S_1$ and atomicity of $M$, note that $\varphi^{n_\varphi}(M_1)=\{1\}$ and $\varphi(M_1)\subseteq M_1$, which we will be using more than once.

\begin{lm}\label{lm3}
If $1\notin\varphi(S)$, then $M_0=\mathrm{Fix}(\varphi^p)$, it is a finitely generated left preGarside monoid, and the restriction of $\nu$ to $M_0$ is additive and homogeneous.
\end{lm}

\begin{proof}
The order of every periodic point of $S$ divides $p=|S|!$, then $M_0\subseteq\mathrm{Fix}(\varphi^p)$. Let $s_1,\dots,s_n\in S$ such that $\varphi^p(s_1\cdots s_n)=s_1\cdots s_n$. In view of the homogeneity and additivity of $\nu$, and the fact that $1\notin\varphi(S)$, we have $\varphi(s_i)\in S$ for all $i$. By (\ref{eq2}), $S_0=S\cap\varphi^p(S)$. Then $\varphi(s_1\cdots s_n)\in\langle S_0\rangle^+$, which means that $s_1\cdots s_n\in M_0$. Thus we have the equality $M_0=\mathrm{Fix}(\varphi^p)$.\\
The submonoid $M_0$ is finitely generated by definition, and it is atomic and left cancellative because $M_0\subseteq M$. The restriction $\nu|_{M_0}$ of $\nu$ is additive, and since $\mathcal{S}(M_0)=S_0$, it is also homogeneous.\\
Let $a,b$ lie in $M_0$ such that the set $\Gamma_0=\{c\in M_0\mid a\preceq c~\text{and}~b\preceq c\}$ is nonempty, and set $\delta=a\vee b$ their lcm in $M$. Set $\psi=\varphi^p$. Since $S_0=S\cap\mathrm{Fix}(\psi)$, we have $\psi|_{M_0}=\mathrm{Id}_{M_0}$. Let $c$ lie in $\Gamma_0$ and $a',a''$ belong to $M$ such that $c=aa'$ and $\delta=aa''$. One has $a,c\in\mathrm{Fix}(\psi)$, then $a'=\psi(a')$ by cancellativity. We have $\delta\preceq c$, then $a''\preceq a'$, and so $\psi^n(a'')\preceq a'$ for all $n\in\N$. The sequence of integers $(\nu(\psi^n(a'')))_{n\in\N}$ is increasing because of the additivity and homogeneity of $\nu$ and the fact that $1\notin\varphi(S)$. On the other hand, it is bounded by $\nu(a')$ because $\psi^n(a'')\preceq a'$ for all $n$. Thus $(\nu(\psi^n(a'')))_{n\in\N}$ is stationary from some rank $m_a\in\N^*$. Write $\psi^{m_a}(a'')=s_1\cdots s_r$ with $s_1,\dots,s_r\in S$. Since $\nu(\psi^{m_a}(a''))=\nu(\psi^{m_a+1}(a''))$, then $\psi(s_i)\in S$ for all $i$. So $\{s_i,\varphi(s_i),\dots,\varphi^p(s_i)\}\subseteq S$, and therefore $\psi(s_i)\in S_0$ for all $i$. Thus $\psi^{m_a+1}(a'')\in M_0$.\\
Similarly, for $b',b''\in M$ such that $c=bb'$ and $\delta=bb''$, we have some rank $m_b\in\N^*$ such that $\psi^{m_b+1}(b'')$ is in $M_0$. The inclusion $M_0\subseteq\mathrm{Fix}(\psi)$ yields $\psi^{m_a+1}(a'')=\psi^{m_a+m_b}(a'')$ and $\psi^{m_b+1}(b'')=\psi^{m_b+m_a}(b'')$. We have $\psi^{m_a+m_b}(\delta)=a\psi^{m_a+1}(a'')=b\psi^{m_b+1}(b'')$, then $\psi^{m_a+m_b}(\delta)\in\Gamma_0$. If $c\in\Gamma_0$, then $\delta\preceq c$, and so $\psi^{m_a+m_b}(\delta)\preceq\psi^{m_a+m_b}(c)=c$. Write $c=\psi^{m_a+m_b}(\delta)c'$ with $c'\in M$. Both $c$ and $\psi^{m_a+m_b}(\delta)$ are in $\mathrm{Fix}(\psi)$, then by cancellativity, $\psi(c')=c'$. Thus, $c'\in M_0$ and $\psi^{m_a+m_b}(\delta)$ is the least element of $\Gamma_0$. Whence property ($b_L$) holds, so $M_0$ is a left preGarside monoid.
\end{proof}~

\begin{lm}\label{lm4}
If $\pi$ is well-defined, then $M_2$ is a finitely generated left preGarside monoid, and the restriction of $\nu$ to $M_2$ is additive and homogeneous.
\end{lm}

\begin{proof}
As in the previous proof, the submonoid $M_2$ is atomic, left cancellative and finitely generated, and the restriction $\nu|_{M_2}$ is additive and homogeneous.\\
Let $a,b$ lie in $M_2$ such that the set $\Gamma_2=\{c\in M_2\mid a\preceq c~\text{and}~b\preceq c\}$ is nonempty, and set $\delta=a\vee b$ their lcm in $M$. Let $c\in\Gamma_2$ and $a',a''\in M$ such that $c=aa'$ and $\delta=aa''$. We have $\delta\preceq c$, then $a''\preceq a'$. So write $a'=a''\hat{a}$ with $\hat{a}\in M$. By the homogeneity and additivity of $\nu$, for all $u\in M$ we have $\nu(\pi(u))\leq\nu(u)$, and ~$\nu(\pi(u))=\nu(u)~\Leftrightarrow~\pi(u)=u~\Leftrightarrow~u\in M_2~$. We have $\nu(c)=\nu(a)+\nu(a'')+\nu(\hat{a})$ and $\nu(\pi(c))=\nu(\pi(a))+\nu(\pi(a''))+\nu(\pi(\hat{a}))$. But $a$ and $c$ belong to $M_2$, then $\nu(a'')+\nu(\hat{a})=\nu(\pi(a''))+\nu(\pi(\hat{a}))$. So $\nu(\pi(a''))=\nu(a'')$ because $\nu(\pi(u))\leq\nu(u)$ for all $u\in M$. Then $a''\in M_2$. Similarly, there is $b''$ in $M_2$ such that $\delta=bb''$. Thus $\delta$ is the least element of $\Gamma_2$, whence property ($b_L$). Therefore $M_2$ is a left preGarside monoid.
\end{proof}~

\begin{lm}\label{lm5}
Let $N_1$ be a left preGarside monoid, and $N_2$ a monoid. Assume there exists a morphism $f:N_1\rightarrow N_2$ that is a retraction. Then $N_2$ is left preGarside.
\end{lm}

\begin{proof}
The morphism $f$ is a retraction, then we have a section $g:N_2\rightarrow N_1$ such that $f\circ g=\mathrm{Id}_{N_2}$. Thus $N_2$ embeds in $N_1$. So $N_2$ is atomic and left cancellative.  Let $a,b\in N_2$ such that the set $\Lambda=\{c\in N_2\mid a\preceq c~\text{and}~b\preceq c\}$ is nonempty. Set $\delta=g(a)\vee g(b)$ the least common right multiple of $g(a)$ and $g(b)$ in $N_1$, and write $\delta=g(a)a'=g(b)b'$ with $a',b'\in N_1$. Thus $f(\delta)=af(a')=bf(b')$. Let $c\in\Lambda$ and $c'\in N_1$ such that $g(c)=\delta c'$. Then $c=f(\delta)f(c')$, and so $f(\delta)$ divides $c$ in $N_2$. Hence, $f(\delta)$ is the least element of $\Lambda$, and therefore $N_2$ is left preGarside.
\end{proof}~

\begin{lm}\label{lm6}
For every $d\in\N^*$, we have $$(\pi\circ\varphi)^d=\pi\circ\varphi^d.$$
\end{lm}

\begin{proof}
If $1\notin\varphi(S)$, then $S_1=\emptyset$ and $\pi=\mathrm{Id}_M$. Assume $1\in\varphi(S)$. We show the result by induction on $d$. The case $d=1$ being trivial, assume  $d>1$ and the result holds for smaller integers. It suffices to check the equality on the generators. By definition of $S_1$, if $a\in S_1$, then $\varphi^d(a)\in M_1=\langle S_1\rangle^+$ because $\varphi(M_1)\subseteq M_1$. Thus, $(\pi\circ\varphi)^d(a)=\pi\circ\varphi^d(a)=1$. If $a\in S_2$, write $\varphi(a)=u_0a_1u_1\cdots a_ku_k$, with $a_1,\dots,a_k\in S_2$ and $u_0,\dots,u_k\in M_1$. On the one hand, we have $$\pi\circ\varphi^d(a)=\pi\circ\varphi^{d-1}(u_0a_1u_1\cdots a_ku_k)=\pi\circ\varphi^{d-1}(a_1\cdots a_k)$$ in view of $\varphi^{d-1}(u_i)\in M_1$ for every $i$. On the other hand, and by the induction hypothesis, $$(\pi\circ\varphi)^d(a)=(\pi\circ\varphi)^{d-1}\circ\pi\circ\varphi(a)= \pi\circ\varphi^{d-1}\circ\pi(u_0a_1u_1\cdots a_ku_k)=\pi\circ\varphi^{d-1}(a_1\cdots a_k).$$ Thus $(\pi\circ\varphi)^d(a)=\pi\circ\varphi^d(a)$ for every $d\in\N^*$.
\end{proof}~

In view of the previous lemmas, we may now prove Theorem \ref{thm1} in two parts, depending on whether $1$ lies in $\varphi(S)$ or not.\\

\begin{proof}[Proof of Theorem \ref{thm1}]
Case I: $1\notin\varphi(S)$. By the equality (\ref{eq3}) in Lemma \ref{lm2}, the morphism $\varphi$ restricts to an endomorphism $\varphi_0$ of $M_0=\langle S_0\rangle^+$. We show that
\begin{equation}\label{eq4}
\mathrm{Fix}(\varphi)=\mathrm{Fix}(\varphi_0).
\end{equation}
It is immediate that $\mathrm{Fix}(\varphi_0)=\mathrm{Fix}(\varphi)\cap M_0$, so it suffices to show that $\mathrm{Fix}(\varphi)\subseteq M_0$. Let $u=a_1\cdots a_k$ belong to $\mathrm{Fix}(\varphi)$, with $a_1,\dots,a_k\in S$. Then $a_1\cdots a_k=u=\varphi^p(u)=\varphi^p(a_1)\cdots\varphi^p(a_k)$. Since $1\notin \varphi(S)$, and $\nu$ is additive and homogeneous, we have $\nu(\varphi^p(a_i))=\nu(a_i)$ and so $\varphi^p(a_i)\in S$ for all $i$. By Lemma \ref{lm2}, we have $S_0=S\cap\varphi^p(S)$, so $\varphi^p(a_i)\in S_0$ for all $i$, and then $u=\varphi^p(a_1)\cdots\varphi^p(a_k)\in M_0$. Therefore, $\mathrm{Fix}(\varphi)\subseteq M_0$, and so $\mathrm{Fix}(\varphi)=\mathrm{Fix}(\varphi_0)$.

Now $\varphi_0|_{S_0}$ is a permutation, and by Lemma \ref{lm3}, $M_0$ is a finitely generated left preGarside monoid. Then by Lemma \ref{lm1} (where in the case of a permutation, the norm of $M_0$ does not have to be additive or homogeneous), $\mathrm{Fix}(\varphi_0)$, and therefore $\mathrm{Fix}(\varphi)$, is a finitely generated left preGarside monoid.\\

Case II: $1\in\varphi(S)$. Denote $n=n_\varphi$, and recall that $\pi=\pi_{S_2}$. Consider the morphism $\varphi_2=(\pi\circ\varphi)|_{M_2}$ that is clearly in $\mathrm{End}(M_2)$. We have $1\notin\varphi_2(S_2)$. Indeed, if $\varphi_2(s)=\pi(\varphi(s))=1$ for some $s\in S_2$, then $\varphi(s)\in M_1$, which means that $\varphi^n(\varphi(s))=\varphi^{n+1}(s)=1$. But since $n=\mathrm{max}\{k\in\N^*\mid\exists s\in S~\mathrm{such~that}~\varphi^k(s)=1~\mathrm{and}~\varphi^{k-1}(s)\neq1\}$, then $\varphi^n(s)=1$, which contradicts the fact that $s\in S_2$. Thus, by Lemma \ref{lm4} and Case I, $\mathrm{Fix}(\varphi_2)$ is a finitely generated left preGarside monoid.

We claim that
\begin{equation}\label{eq5}
\mathrm{Fix}(\varphi)=\varphi^n(\mathrm{Fix}(\varphi_2)).
\end{equation}
As seen before, we have $\varphi^n(M_1)=\{1\}$ and $\varphi(M_1)\subseteq M_1$. Let $u\in\mathrm{Fix}(\varphi)$. We may factor $u=u_0a_1u_1\cdots a_ku_k$, with $a_1,\dots,a_k\in S_2$ and $u_0,\dots,u_k\in M_1$. It follows that $u=\varphi^n(u)=\varphi^n(a_1a_2\cdots a_k)$. For every $i$, we have $\pi\circ\varphi(u_i)=1$, in view of $\varphi(M_1)\subseteq M_1$. Now $a_1a_2\cdots a_k\in M_2$, and $$\varphi_2(a_1a_2\cdots a_k)=\pi\circ\varphi(a_1a_2\cdots a_k)=\pi\circ\varphi(u_0a_1u_1\cdots a_ku_k)=\pi(u)=a_1a_2\cdots a_k.$$ Hence $a_1a_2\cdots a_k\in\mathrm{Fix}(\varphi_2)$, and so $u=\varphi^n(a_1a_2\cdots a_k)\in\varphi^n(\mathrm{Fix}(\varphi_2))$. Thus $\mathrm{Fix}(\varphi)\subseteq\varphi^n(\mathrm{Fix}(\varphi_2))$.

Conversely, let $v=a_1a_2\cdots a_k\in\mathrm{Fix}(\varphi_2)$, with $a_1,\dots,a_k\in S_2$. Clearly,
\begin{equation}\label{eq6}
\varphi^n\circ\pi=\varphi^n.
\end{equation}
Hence $v=\varphi_2(v)=\pi\circ\varphi(v)$ yields $\varphi(\varphi^n(v))=\varphi^n\circ\varphi(v)=
\varphi^n\circ\pi\circ\varphi(v)=\varphi^n(\pi\circ\varphi(v))=\varphi^n(v)$ and so $\varphi^n(v)\in\mathrm{Fix}(\varphi)$. Thus $\varphi^n(\mathrm{Fix}(\varphi_2))\subseteq\mathrm{Fix}(\varphi)$ and so $\mathrm{Fix}(\varphi)
=\varphi^n(\mathrm{Fix}(\varphi_2))$.\\

In view of (\ref{eq5}), we have a morphism $f:=\varphi^n|_{\mathrm{Fix}(\varphi_2)}:\mathrm{Fix}(\varphi_2)\rightarrow\mathrm{Fix}(\varphi)$. Let $u\in\mathrm{Fix}(\varphi)$. By Lemma \ref{lm6}, we get $\varphi_2(\pi(u))=\pi\circ\varphi\circ\pi(u)= \pi\circ\varphi\circ\pi\circ\varphi(u)=\pi\circ\varphi^2(u)=\pi(u)$. Then $\pi(\mathrm{Fix}(\varphi))\subseteq\mathrm{Fix}(\varphi_2)$, and so we have another morphism $g:=\pi|_{\mathrm{Fix}(\varphi)}:\mathrm{Fix}(\varphi)\rightarrow\mathrm{Fix}(\varphi_2)$. In view of (\ref{eq6}), one has $\varphi^n\circ\pi(u)=\varphi^n(u)=u$ for every $u\in\mathrm{Fix}(\varphi)$, and then $f\circ g=\mathrm{Id}|_{\mathrm{Fix}(\varphi)}$. So the morphism $f$ is a retraction with section $g$. We established that $\mathrm{Fix}(\varphi_2)$ is a left preGarside monoid. So by Lemma \ref{lm5}, $\mathrm{Fix}(\varphi)$ is also a left preGarside monoid. The submonoid $\mathrm{Fix}(\varphi_2)$ is finitely generated, then so is $\mathrm{Fix}(\varphi)$, in view of (\ref{eq5}).
\end{proof}

We will see in the proof of Proposition \ref{prp2} that the equality (\ref{eq5}) induces an isomorphism between $\mathrm{Fix}(\varphi)$ and $\mathrm{Fix}(\varphi_2)$.

\begin{ex}\label{ex4}
Let $M=\langle a,b,c\mid abab=baba,~ac=ca\rangle^+$, and $\varphi\in\mathrm{End}(M)$ such that $\varphi(a)=b$, $\varphi(b)=a$ and $\varphi(c)=1$, with $\nu$ additive and $\nu(a)=\nu(b)=\nu(c)=1$. By using the notations above, we have $n=1$, $S_1=\{c\}$, $S_2=\{a,b\}$, $M_2=\langle a,b\mid abab=baba\rangle^+$ and $\varphi_2\in\mathrm{End}(M_2)$, such that $\varphi_2(a)=b$ and $\varphi_2(b)=a$. Then one has $\mathrm{Fix}(\varphi_2)=\langle abab\rangle^+$, and $\mathrm{Fix}(\varphi)=\varphi(\mathrm{Fix}(\varphi_2))=\mathrm{Fix}(\varphi_2 )=\langle abab\rangle^+$.
\end{ex}~

\subsection{The submonoid of periodic points}

As in the previous subsection, we consider a finitely generated left preGarside monoid $M$, equiped with an additive and homogeneous norm $\nu$, and we fix $\varphi\in\mathrm{End}(M)$. We also set $n=n_\varphi$, $S=\mathcal S(M)$, $S_0=S\cap\mathrm{Per}(\varphi)$, $S_1=S\cap(\varphi^n)^{-1}\{1\}$, $S_2=S\setminus S_1$, $p=|S|!$, and $\pi:=\pi_{S_2}$ when it is well-defined.\\

\begin{prp}\label{prp1}
If the morphism $\pi$ is well-defined, then we have $$\mathrm{Per}(\varphi)=\mathrm{Fix}(\varphi^{pn}).$$
\end{prp}

The proof of this proposition is also in the spirit of \cite[Theorem 3.2]{RoS}, where the particular case of trace monoids was considered.

\begin{proof}
Case I: $1\notin\varphi(S)$. By definition, $n=1$ in this case. We will use induction on $|S|$. The case $|S|=0$ being trivial, assume that $|S|>0$ and the result holds for smaller sets.

We may assume $S_0\subsetneq S$, otherwise $\varphi|_S$ would be a permutation, and since the order of $\varphi|_S$ must divide the order of the symmetric group on $S$, which is $p$, we would get $(\varphi|_S)^p=\mathrm{Id}_S$ and therefore $\varphi^p=\mathrm{Id}_M$, yielding $\mathrm{Fix}(\varphi^p)=M=\mathrm{Per}(\varphi)$.

For every $r\in\mathbb{N}^*$, if we replace $\varphi$ by $\varphi^r$, then $S_0$ remains the same in view of $\mathrm{Per}(\varphi)=\mathrm{Per}(\varphi^r)$, and so does $M_0$. On the other hand, by (\ref{eq3}), we restrict $\varphi$ to $\varphi_0=\varphi|_{M_0}$, and we have $\varphi^r|_{M_0}=(\varphi|_{M_0})^r=\varphi_0^r$. Hence
\begin{equation}\label{eq7}
\mathrm{Fix}(\varphi^r)=\mathrm{Fix}(\varphi_0^r)
\end{equation}
by applying (\ref{eq4}) to $\varphi^r$. By the induction hypothesis and Lemma \ref{lm3}, we have $\mathrm{Per}(\varphi_0)=\mathrm{Fix}(\varphi_0^{|S_0|!})$. Since $|S_0|!$ divides $p$, we get $\mathrm{Per}
(\varphi_0)=\mathrm{Fix}(\varphi_0^{|S_0|!})\subseteq\mathrm{Fix}(\varphi_0^p)\subseteq\mathrm{Per}(\varphi_0)$ and so $\mathrm{Per}(\varphi_0)=\mathrm{Fix}(\varphi_0^p)$.
Together with (\ref{eq7}), this yields $$\mathrm{Per}(\varphi)=\cup_{r\geq1}\mathrm{Fix}(\varphi
^r)=\cup_{r\geq1}\mathrm{Fix}(\varphi_0^r)=\mathrm{Per}(\varphi_0)=\mathrm{Fix}(\varphi_0^p)=\mathrm{Fix}(\varphi^p)$$ as required.\\

Case II: $1\in\varphi(S)$. By definition, we have $\mathrm{Per}(\varphi)\supseteq\mathrm{Fix}(\varphi^{pn})$. Conversely, let $u\in\mathrm{Per}(\varphi)$, say $u\in\mathrm{Fix}(\varphi^r)$. We may factor $u=u_0a_1u_1\cdots a_ku_k$, with $a_1,\dots,a_k\in S_2$ and $u_0,\dots,u_k\in M_1$. It follows that $u=\varphi^{rn}(u)=\varphi^{rn}(a_1a_2\cdots a_k)$. Now $a_1a_2\cdots a_k\in M_2$, and Lemma \ref{lm6} yields $a_1a_2\cdots a_k=\pi\circ\varphi^{rn}(a_1a_2\cdots a_k)=(\pi\circ\varphi)^
{rn}(a_1a_2\cdots a_k)$. Consequently $a_1a_2\cdots a_k$ belongs to $\mathrm{Fix}(\varphi_2^{rn})\subseteq\mathrm{Per}(\varphi_2)$. As in the proof of Theorem \ref{thm1}, we have $1\notin\varphi_2(S_2)$. Thus, by Lemma \ref{lm4} and Case I, we have $\mathrm{Per}(\varphi_2)=\mathrm{Fix}(\varphi_2^{|S_2|!})$. We get $a_1a_2\cdots a_k\in\mathrm{Fix}(\varphi_2^{|S_2|!})\subseteq\mathrm{Fix}(\varphi_2^{pn})$, and so $a_1a_2\cdots a_k=\pi\circ\varphi^{pn}(a_1a_2\cdots a_k)$ in view of Lemma \ref{lm6}. Hence $\varphi^{pn}(u)=\varphi^{pn}(a_1a_2\cdots a_k)=v_0a_1v_1\cdots a_kv_k$ for some $v_0,v_1,\dots,v_k$ in $M_1$. Thus $$\varphi^{2pn}(u)=\varphi^{pn}\circ\pi\circ\varphi^{pn}(u)=
\varphi^{pn}\circ\pi(v_0a_1v_1\cdots a_kv_k)=\varphi^{pn}\circ\pi(u_0a_1u_1\cdots a_ku_k)=\varphi^{pn}(u).$$ Since $\varphi^r(u)=u$, this yields to $u=\varphi^r(u)=\varphi^{2r}(u)=\dots=\varphi^{pnr}(u)=\varphi^{pn(r-1)}(u)=\dots=\varphi^{pn}(u)$. Therefore $\mathrm{Per}(\varphi)=\mathrm{Fix}(\varphi^{pn})$.
\end{proof}~

\begin{thm}\label{thm2}
If the morphism $\pi$ is well-defined, then $\mathrm{Per}(\varphi)$ is also a finitely generated left preGarside monoid.
\end{thm}

\begin{proof}
In Proposition \ref{prp1}, we showed that $\mathrm{Per}(\varphi)=\mathrm{Fix}(\psi)$, where $\psi=\varphi^{pn}$. Denote $S_1(\varphi)=S_1$, $S_1(\psi)=S\cap(\psi^{n_\psi})^{-1}\{1\}$, $S_2(\varphi)=S_2$, $S_2(\psi)=S\setminus S_1(\psi)$, $\pi(\varphi)=\pi$ and $\pi(\psi)=\pi_{S_2(\psi)}$. We have $S_1(\varphi)=\{s\in S\mid\exists k\in\N~\mathrm{such~that}~\varphi^k(s)=1\}=\{s\in S\mid\exists k\in\N~\mathrm{such~that}~\psi^k(s)=1\}=S_1(\psi)$. Then $S_2(\psi)=S_2(\varphi)$, and so $\pi(\psi)=\pi(\varphi)=\pi$. Thus we can apply Theorem \ref{thm1} to $\psi$, which means that $\mathrm{Fix}(\psi)$, and therefore $\mathrm{Per}(\varphi)$, is a finitely generated left preGarside monoid.
\end{proof}~

\subsection{The case of Artin monoids}

A symmetry of an Artin group $A$ generated by $S$, is an endomorphism $\varphi$ of $A$ such that $\varphi_{|S}$ is a permutation. In \cite[Lemma 10]{Cri2} and \cite[Corollary 4.4]{Mich}, it is shown that, given a group $G$ of symmetries of an Artin group $A$, the submonoid of elements fixed by $G$, is isomorphic to another Artin monoid. In particular, given an Artin monoid M generated by $S$, and $\varphi\in\mathrm{End}(M)$ such that $\varphi_{|S}$ is a permutation (i.e. $\varphi\in\mathrm{Aut}(M)$), the submonoid $\mathrm{Fix}(\varphi)$ is also an Artin monoid. Below, we will show that this is also the case for $\mathrm{Per}(\varphi)$, and for every $\varphi\in\mathrm{End}(M)$.\\

Let $M=\langle S\mid[a,b\rangle^{m_{ab}}=[b,a\rangle^{m_{ab}};~m_{ab}\neq\infty\rangle^+$ be an Artin monoid, and $\varphi$ be in $\mathrm{End}(M)$. By \cite{BrS}, Artin monoids satisfy all properties of preGarside monoids. The set of atoms $\mathcal S(M)$ of $M$ is $S$, and the length $\ell_S$ is an additive and homogeneous norm over $M$. Thus, we can apply the results from the previous subsections. \\
As before, set $n=n_\varphi$, $S_0=S\cap\mathrm{Per}(\varphi)$, $S_1=S\cap(\varphi^n)^{-1}\{1\}$, $S_2=S\setminus S_1$, $p=|S|!$, and $\pi:=\pi_{S_2}$ when it is well-defined. It is known that the submonoids $M_0=\langle S_0\rangle^+$, $M_1=\langle S_1\rangle^+$ and $M_2=\langle S_2\rangle^+$ are Artin monoids too.\\

\begin{lm}\label{lm7}
The morphism $\pi$ is well-defined. 
\end{lm}

\begin{proof}
If $1\notin\varphi(S)$, then $S_2=S$ and $\pi=\mathrm{Id}_M$. Suppose $1\in\varphi(S)$. It suffices to verify that $\pi([a,b\rangle^{m_{ab}})=\pi([b,a\rangle^{m_{ab}})$ for all $m_{ab}\neq\infty$. Let $a,b\in S$ such that $m_{ab}\neq\infty$. If $m_{ab}$ is even, or if $a$ and $b$ are both in $S_1$ or in $S_2$, the equality holds trivially. Suppose we have $m_{ab}=2k+1$ for some $a\in S_1$, $b\in S$, and $k>0$. Then $\varphi^n([a,b\rangle^{m_{ab}})=(\varphi^n(b))^k$ and $\varphi^n([b,a\rangle^{m_{ab}})=(\varphi^n(b))^{k+1}$. Thus, by cancellativity, $\varphi^n(b)=1$, so $b\in S_1$ and we are done as remarked above.
\end{proof}~

\begin{prp}\label{prp2}
Let $M$ be an Artin monoid, and $\varphi$ be in $\mathrm{End}(M)$. Then the submonoids $\mathrm{Fix}(\varphi)$ and $\mathrm{Per}(\varphi)$ are also Artin monoids.
\end{prp}

\begin{proof}
Assume first $1\notin\varphi(S)$. In the proof of Theorem \ref{thm1}, we showed that $\mathrm{Fix}(\varphi)=\mathrm{Fix}(\varphi_0)$, with $\varphi_0\in\mathrm{End}(M_0)$ and $\varphi|_{S_0}$ is a permutation. Then by \cite[Lemma 10]{Cri2}, $\mathrm{Fix}(\varphi_0)$, and therefore $\mathrm{Fix}(\varphi)$, is an Artin monoid.\\
In this case, $n=1$, and by Proposition \ref{prp1}, we have $\mathrm{Per}(\varphi)=\mathrm{Fix}(\varphi^{p})$. Since $1\notin\varphi(S)$, then $\ell_S(\varphi(u))\geq\ell_S(u)$ for all $u\in M$, and so $1\notin\varphi^p(S)$. Thus, $\mathrm{Fix}(\varphi^{p})$, and therefore $\mathrm{Per}(\varphi)$, is again an Artin monoid.

Assume now $1\in\varphi(S)$. In the proof of Theorem \ref{thm1}, we showed that $\mathrm{Fix}(\varphi)=\varphi^n(\mathrm{Fix}(\varphi_2))$, with $\varphi_2\in\mathrm{End}(M_2)$ and $1\notin\varphi_2(S_2)$. Let $u,v\in\mathrm{Fix}(\varphi_2)$ such that $\varphi^n(u)=\varphi^n(v)$. Then $\pi\circ\varphi^n(u)=\pi\circ\varphi^n(v)$, and so, by Lemma \ref{lm6}, $(\pi\circ\varphi)^n(u)=(\pi\circ\varphi)^n(v)$. Thus $u=\varphi_2^n(u)=\varphi_2^n(v)=v$. Hence, the morphism $\varphi^n|_{\mathrm{Fix}(\varphi_2)}:\mathrm{Fix}(\varphi_2)\rightarrow\mathrm{Fix}(\varphi)$ is not only surjective, but also injective. Therefore, $\mathrm{Fix}(\varphi)$ is isomorphic to $\mathrm{Fix}(\varphi_2)$. By Case I, $\mathrm{Fix}(\varphi_2)$ is an Artin monoid, then so is $\mathrm{Fix}(\varphi)$.\\
By Proposition \ref{prp1}, we have $\mathrm{Per}(\varphi)=\mathrm{Fix}(\varphi^{pn})$. Since $1\in\varphi(S)$, one has $1\in\varphi^{pn}(S)$. Thus $\mathrm{Fix}(\varphi^{pn})$, and so $\mathrm{Per}(\varphi)$, is an Artin monoid.
\end{proof}~

\section{Inequalities and some equivalences between metrics}

The purpose of this section is to define three metrics $d_1$, $d_2$ and $d_3$ in finitely generated preGarside monoids, to compare them in general, and in the particular cases of trace monoids and Garside monoids.

\subsection{Metrics and normal forms}

In order to define our three metrics, we start by introducing the following general framework. Recall that given a non empty set $X$, we denote by $X^*$ the set of all finite words over $X$. Henceforth, these words will be denoted as tuples, to avoid any confusion with the monoids elements. Let $M$ be a monoid, $X$ be a non empty set, and $\iota:M\hookrightarrow X^*$ be an injective map. For $u,v\in M$ with $\iota(u)=(u_1,\dots,u_n)$ and $\iota(v)=(v_1,\dots,v_m)$, we define 
\[ r(u,v) = \left \{
\begin{array}{ll}
     \mathrm{max}\{k\geq0\mid u_1=v_1,\dots,u_k=v_k\} & \text{if $u\neq v$}\\
     \infty & \text{if $u=v$}\\
\end{array}.
\right.\]
The metric $d$ over $M$, associated to $\iota$, is defined, for all $u,v\in M$, by $$d(u,v)=2^{-r(u,v)}.$$ When $\iota(u)=(u_1,\dots,u_n)$ for some $u\in M$, then for all $k\leq n$, we denote $\iota^{[k]}(u)=(u_1,\dots,u_k)\in X^*$.\\

Let $M$ be a finitely generated preGarside monoid. For each metric $d_i$ over $M$, we will define $X_i$, $\iota_i$ and $r_i$ as above. The set $X_1$ for the first distance $d_1$ is defined in \cite{BDM}, where it is denoted by $P$; the subset of $M$ with a {\em preGarside structure}. It contains the finite set of atoms $S=\mathcal S(M)$, and whenever it contains an element, it also contains all its left and right divisors (\cite[Proposition 2.4]{BDM}). We will denote it by $M_{\mathrm{red}}$, since in the case of an Artin monoid, it is just the set of reduced elements, that we will recall bellow. The properties of $M_{\mathrm{red}}$ shown in \cite{Mich} for Artin monoids, hold in finitely generated preGarside monoids with the same proofs, as stated in \cite{BDM}. Namely (\cite[Proposition 2.12]{BDM}), there is a unique function $\alpha:M\rightarrow M_{\mathrm{red}}$ which induces the identity on $M_{\mathrm{red}}$, and satisfies
\begin{equation}\label{eq8}
\alpha(uv)=\alpha(u\alpha(v)),
\end{equation}
for all $u,v\in M$. Further, $\alpha(u)$ is the unique maximal element (for $\preceq$) in the set $\{v\in M_{\mathrm{red}}\mid v\preceq u\}$.\\
Let $M=\langle S\mid[a,b\rangle^{m_{ab}}=[b,a\rangle^{m_{ab}}; m_{ab}\neq\infty\rangle^+$ be an Artin monoid, whose natural length function is denoted, as in the preliminaries, by $\lS$. And let $W=\langle S\mid (ab)^{m_{ab}}=1; m_{ab}\neq\infty\rangle$ be the corresponding Coxeter group. There is also a length function on $W$ (see \cite{Mich}), which we denote also by $\lS$. It is known that two minimal expressions of an element of $W$ are equivalent by using Artin relations only. The length of an element is defined by the length of any of its minimal expressions as products of elements of $S$. This implies that the induced quotient map from M to W has a canonical section (as a map of sets), whose image $M_{\mathrm{red}}$ consists of those elements of $M$ which have the same length as their image in $W$.

Let $M$ be a finitely generated preGarside monoid. To every element of $M$, can be associated a (left) {\em normal form (n.f)}, that is called the (left) {\em greedy normal form}, and defined as follows. To $1_M$, we associate the empty sequence. And for $u\in M\setminus \{1\}$ and $u_1,\dots,u_n\in M_{\mathrm{red}}$, we say that $u=u_1\cdots u_n$ is in normal form (n.f), if and only if no $u_i$ is equal to 1 and for any $i$ we have $u_i=\alpha(u_i\cdots u_n)$. In view of (\ref{eq8}), the normality of a form can be seen locally (\cite[Proposition 2.21]{BDM}): $u_1\cdots u_k$ is a normal form if and only if $u_iu_{i+1}$ is for all $i$. This implies that any segment $u_i\cdots u_j$ of a normal form is normal. For $u=u_1\cdots u_n$ (n.f), we define $\iota_1(u)=(u_1,\dots,u_n)$, and denote $n=|u|_1$. Let $u,v\in M$ with $\iota_1(u)=(u_1,\dots,u_n)$ and $\iota_1(v)=(v_1,\dots,v_m)$. We define $r_1(u,v)$ exactly as $r(u,v)$ above. Using the convention $2^{-\infty}=0$, the metric $d_1$ is defined by $$d_1(u,v)=2^{-r_1(u,v)}.$$

Another important normal form $\iota_2$ over $M$, that we call the {\em Foata normal form}, is defined as follows. Let $X_2=\{u\in M\mid\exists T\subseteq S, u=\Delta(T)\}$, where $S=\mathcal S(M)$ and $\Delta(T)$ is the least right common multiple of the elements of $T$, which exists if and only if there is a right common multiple (\cite[Lemma 2.1]{GoP}). For $u\in M\setminus \{1\}$, there exists a unique $\iota_2(u)=(u_1,\dots,u_n)\in X_2^*$ such that $u=u_1\cdots u_n$ and $u_i=\Delta(\{s\in S\mid s\preceq u_i\cdots u_n\})$. When $\iota_2(u)=(u_1,\dots,u_n)$, we denote $n=|u|_2$. And similarly, the metric $d_2$ associated to $\iota_2$, is known as the {\em FNF metric}, and defined in \cite{BMP}, for all $u,v\in M$, by $$d_2(u,v)=2^{-r_2(u,v)}.$$

When the monoid $M$ is equiped with an additive and homogeneous norm $\nu$, we can assume that $\nu(s)=1$ for all $s$ in $S$, call this norm the length over $S$, and denote it by $\ell_S$. In this case, and in addition to $d_1$ and $d_2$, there is a third and useful metric, decribed in \cite{RoS} for the particular case of trace monoids, that we will denote by $d_3$. Given $u,v\in M$, we say that $v$ is a {\em prefix} of $u$, when $v$ left-divides $u$. For every $n\in\N$, denote by $\mathrm{Pref}_n(u)$ the set of all prefixes of $u$ of length $n$. Let $X_3=\mathcal P(M)$ be the set of all parts of $M$. For $u\in M\setminus\{1\}$, set $\iota_3(u)=(u_1,\dots,u_n)$ with $n=\ell_S(u)$ and $u_i=\mathrm{Pref}_i(u)$ for all $i$. And set $\iota_3(1)=(\{1\})$. Then the metric $d_3$, known as the {\em prefix metric}, is defined in \cite{Kwi} as above, for all $u,v\in M$, by $$d_3(u,v)=2^{-r_3(u,v)}.$$ Note that in this case, for $u,v\in M$, we have $r_3(u,v)=\max\{n\in\N\mid\mathrm{Pref}_n(u)= \mathrm{Pref}_n(v)\}$, because for $1\leq n\leq\ell_S(u)$, $$\mathrm{Pref}_n(u)=\mathrm{Pref}_n(v)~\Leftrightarrow~ \mathrm{Pref}_k(u)=\mathrm{Pref}_k(v),~ \forall1\leq k\leq n.$$

\subsection{Relations between $d_1$, $d_2$ and $d_3$}

In this subsection, we will compare the first distance $d_1$ with the other two for a finitely generated preGarside monoid $M$, equiped with a length $\ell_S$. We start with $d_1$ and $d_3$.

\begin{lm}\label{lm8}
Let $u$ lie in $M$. Set $\iota_1(u)=(u_1,\dots,u_n)$. Then, for $1\leq k\leq n$, we have $$\mathrm{Pref}_k(u)=\mathrm{Pref}_k(u_1\cdots u_k).$$
\end{lm}

\begin{proof}
It suffices to show that if $v\preceq u$ with $\ell_S(v)=k$, then $v\preceq u_1\cdots u_k$, which we do by induction on $k$. Assume first $k=1$. If $\lS(v)=1$, then $v\in S\subseteq M_{\mathrm{red}}$. So, by definition of the greedy normal form, $v\preceq\alpha(u)=u_1$. Assume now $k>1$ plus the induction hypothesis. Consider $v$ in $M$ such that $v\preceq u$ and $\lS(v)=k$. Write $v=v's$ and $u=vw$ with $s$ in $S$ and $w$ in $M$. By the induction hypothesis, $v'$ left divides $u_1\cdots u_{k-1}$. Write $u_1\cdots u_{k-1}=v'v''$ with $v''\in M$. Since $v'sw=u_1\cdots u_n=v'v''u_k\cdots u_n$, we have $s\preceq v''u_k\cdots u_n$. But $s\in M_{\mathrm{red}}$, therefore, $s\preceq\alpha(v''u_k\cdots u_n)$. By (\ref{eq8}), $\alpha(v''u_k\cdots u_n)=\alpha(v''\alpha(u_k\cdots u_n))=\alpha(v''u_k)$, so $s\preceq v''u_k$. Hence, $v=v's\preceq v'v''u_k=u_1\cdots u_k$.
\end{proof}~

\begin{prp}\label{prp3}
Let $M$ be a finitely generated preGarside monoid, equiped with a length $\ell_S$. Then we have $$d_3\leq d_1.$$
\end{prp}

\begin{proof}
Consider $u$ and $v$ distinct in $M$. Set $\iota_1(u)=(u_1,\dots,u_n)$ and $\iota_1(v)=(v_1,\dots,v_m)$. If $r_1(u,v)=0$, then $r_3(u,v)\geq r_1(u,v)$, so $d_3(u,v)\leq d_1(u,v)$. Otherwise, for $1\leq k\leq r_1(u,v)$, by Lemma \ref{lm8} we have $$\mathrm{Pref}_k(u)=\mathrm{Pref}_k(u_1\cdots u_k)=\mathrm{Pref}_k(v_1\cdots v_k)=\mathrm{Pref}_k(v).$$ Thus, $r_3(u,v)\geq r_1(u,v)$ and so $d_3(u,v)\leq d_1(u,v)$.
\end{proof}~

We turn now to $d_1$ and $d_2$. Note that the existence of a length $\ell_S$ is only necessary for $d_3$, and we do not need it to compare $d_1$ and $d_2$. The inclusion $X_2\subseteq M_{\mathrm{red}}$, deduced from \cite[Proposition 2.19]{BDM}, will be useful for us.

\begin{lm}\label{lm9}
Let $u$ be in $M$. Set $\iota_1(u)=(u_1,\dots,u_n)$ and $\iota_2(u)=(u'_1,\dots,u'_m)$. Then 
\begin{enumerate}
\item $n\leq m$, and for all $i\leq n$, we have $u'_1\cdots u'_i\preceq u_1\cdots u_i$.
\item For $i\leq n$ and $j\leq m$, if $u'_1\cdots u'_j\preceq u_1\cdots u_i$, then $\iota^{[j]}_2(u_1\cdots u_i)=(u'_1,\dots,u'_j)$.
\end{enumerate}
\end{lm}

\begin{proof}
\begin{enumerate}
\item The fact that $n\leq m$ is a consequence of \cite[Proposition 4.8]{Mich}, also true for preGarside monoids, as stated in \cite{BDM}. On the other hand, by using $X_2\subseteq M_{\mathrm{red}}$ and the same proof as \cite[Proposition 4.10]{Mich}, we get $u'_1\cdots u'_i\preceq u_1\cdots u_i$ for all $i\leq n$.




\item Set $i\leq n$. Assume $u'_1\cdots u'_j\preceq u_1\cdots u_i$. We prove that $\iota^{[j]}_2(u_1\cdots u_i)=(u'_1,\dots,u'_j)$ by induction on $j$. Since $\iota^{[1]}_2(u)=(u'_1)$ and $u'_1\preceq u_1\cdots u_i\preceq u$, we have $\iota^{[1]}_2(u_1\cdots u_i)=(u'_1)$. Hence, the property is true for $j=1$. Assume $j\geq2$ plus the induction hypothesis. Write $u_1\cdots u_i=u'_1\cdots u'_jv$ with $v\in M$. By the induction hypothesis, we have $\iota^{[j]}_2(u_1\cdots u_i)=(u'_1,\dots,u'_{j-1},\iota^{[1]}_2(u'_jv))$. On the other hand, $\iota^{[1]}_2(u'_j\cdots u'_m)=(u'_j)$ and $u'_j\cdots u'_m=u'_jxu_{i+1}\cdots u_n$. This imposes $\iota^{[1]}_2(u'_jx)=(u'_j)$. Therefore $\iota^{[j]}_2(u_1\cdots u_i)=(u'_1,\dots,u'_j)$.
\end{enumerate}
\end{proof}

\begin{prp}\label{prp4}
Let $M$ be a finitely generated preGarside monoid. Then we have $$d_2\leq d_1.$$
\end{prp}

\begin{proof}
Let $u,v$ be in $M$ and distinct. Set $\iota_1(u)=(u_1,\dots,u_n)$, $\iota_1(v)=(v_1,\dots,v_{n'})$, $\iota_2(u)=(u'_1,\dots,u'_m)$ and $\iota_2(v)=(v'_1,\dots,v'_{m'})$. By Lemma \ref{lm9}, $n\leq m$ and $n'\leq m'$. If $m=r_2(u,v)$ or $m'=r_2(u,v)$, then $r_1(u,v)\leq\min(n,n')\leq\min(m,m')=r_2(u,v)$. Therefore $d_2(u,v)\leq d_1(u,v)$. So assume $m<r_2(u,v)$ and $m'<r_2(u,v)$ and set $k=r_2(u,v)$. By assumption $u'_{k+1}\neq v'_{k+1}$. We can therefore assume without restriction that $u'_{k+1}$ does not left divide $v'_{k+1}$. By Lemma \ref{lm9} $i.$ and $ii.$, $\iota^{[k+1]}_2(v_1\cdots v_{k+1})=(v'_1,\dots,v'_{k+1})$. Since $\iota^{[k+1]}_2(u'_1\cdots u'_{k+1})=(u'_1,\dots,u'_{k+1})$ and $u'_{k+1}$ does not left divide $v'_{k+1}$, it follows from Lemma \ref{lm9} that $u'_1\cdots u'_{k+1}$ does not left divide $v_1\cdots v_{k+1}$. But on the other hand, $u'_1\cdots u'_{k+1}$ left divides $u_1\cdots u_{k+1}$ by Lemma \ref{lm9} $i.$. Thus, $v_1\cdots v_{k+1}\neq u_1\cdots u_{k+1}$ and $r_1(u,v)\leq k$. Hence $r_1(u,v)\leq r_2(u,v)$ and $d_2(u,v)\leq d_1(u,v)$.
\end{proof}

\begin{df}\label{df1}
A mapping $\varphi:(X,d)\rightarrow(X',d')$ between metric spaces is {\em uniformly continuous} if $$\forall\varepsilon>0,~\exists\delta>0,~\forall x,y\in X_1:~(d(x,y)<\delta\Rightarrow d'(\varphi(x),\varphi(y))<\varepsilon).$$

If the identity mappings between $(X,d)$ and $(X,d')$ are uniformly continuous, we say that the metrics $d$ and $d'$ are {\em uniformly equivalent}. It is immediate that two equivalent metrics are consequently uniformly equivalent
\end{df}~

Below, we will show that in Garside monoids, $d_1$ is uniformly equivalent to $d_3$, and in trace monoids, $d_2$ and $d_3$ are uniformly equivalent. However, these metrics are not equivalent, nor uniformly equivalent in general. Here are some examples to illustrate that.

\begin{ex}~\label{ex5}
\begin{enumerate}
\item In an Artin monoid $M$ with $m_{ab}=\infty$ for some $a,b\in S$, the metric $d_1$ is not uniformly equivalent (and so not equivalent) to $d_2$, nor to $d_3$. Indeed, write $u_n=(ab)^n$ and $v_n=(ab)^{n+1}$. By definition of the metrics, we have $r_2(u_n,v_n)=r_3(u_n,v_n)=2n$, and since $u_n,v_n\in M_{\mathrm{red}}$, then $r_1(u_n,v_n)=0$. Thus we have $d_1(u_n,v_n)=1$ for all $n$, and $\mathrm{lim}_{n\rightarrow\infty}d_2(u_n,v_n)=\mathrm{lim}_{n \rightarrow\infty}d_3(u_n,v_n)=\mathrm{lim}_{n\rightarrow\infty}2^{-2n}=0$. So $d_1$ cannot be uniformly equivalent to $d_2$ or $d_3$.
\item The metrics $d_2$ and $d_3$ are not equivalent in general. Indeed, Let $M$ be an Artin monoid, with $a,b,c\in S$ such that $m_{ab}=2$ and $m_{ac},m_{bc}\geq3$. Consider $u_n=(ab)^n$ and $v_n=(ab)^nc$. Then $r_2(u_n,v_n)=n$ and $r_3(u_n,v_n)=2n$. We have $$\exists C>0,~d_2\leq Cd_3~~\Leftrightarrow~~\exists C>0,~-r_2\leq\mathrm{log}_2(C)-r_3~~\Leftrightarrow~~\exists C'\in\R,~r_3\leq r_2+C'.$$ But in our example, $r_3=2n$ and $r_2=n$, so there is no $C'$ in $\R$ such that $r_3\leq r_2+C'$. Therefore, $d_2$ and $d_3$ cannot be equivalent.
\item In a non-abelian Artin monoid $M$, the metric $d_1$ is not equivalent to $d_2$, nor to $d_3$. Indeed, let $a,b\in S$ such that $m_{ab}\geq3$. Consider $u_n=(abba)^n$ and $v_n=(abba)^{n+1}$. Then $r_1(u_n,v_n)=2n$ and $r_2(u_n,v_n)=r_3(u_n,v_n)=4n$. Or, as in ii., the existence of some $C>0$ such that $d_1\leq Cd_2$ and $d_1\leq Cd_3$, means there is a $C'$ in $\R$ such that $r_3\leq r_1+C'$ and $r_2\leq r_1+C'$, which is impossible for our example.
\end{enumerate}
\end{ex}~

\subsection{The case of trace monoids}

In this subsection, we focus on {\em right angled Artin monoids (RAAM)}. A RAAM, or a {\em trace monoid}, is an Artin monoid $M=\langle S\mid[a,b\rangle^{m_{ab}}=[b,a\rangle^{m_{ab}}; m_{ab}\neq\infty\rangle^+$, where $m_{ab}\in\{2,\infty\}$ for all $a,b\in S$. Our objective is to obtain a complete comparison of $d_1$, $d_2$ and $d_3$ for trace monoids. Here we prove :
~~\\
~~\\

\begin{prp}\label{prp5}
Assume $M$ is a trace monoid. Then
\begin{enumerate}
\item $d_2$ and $d_3$ are uniformly equivalent.
\item if $M$ is not free abelian, $d_1$ is not uniformly equivalent to $d_2$, nor to $d_3$.
\item if $M$ is the free abelian monoid, $d_1=d_2=d_3$.
\end{enumerate}
\end{prp}~

Indeed, point $i.$ was already proved in \cite{KK} by a topological argument. We provide an algebraic one. Let us start with the following remark :

\begin{rmq}\label{rmq1}
If $M\simeq\mathrm{F}^+_p$ is the free monoid, i.e. $m_{ab}=\infty$ for all $a,b\in S$, then $d_2=d_3$.
\end{rmq}

\begin{proof}
Let $u,v$ lie in $M$. Set $\iota_2(u)=(u_1,\dots,u_n)$ and $\iota_2(v)=(v_1,\dots,v_m)$. For all $a,b\in S$, $m_{ab}=\infty$. Then $\{s\in S\mid s\preceq u_i\cdots u_n\}=\{u_i\}$, $\{s\in S\mid s\preceq v_i\cdots v_m\}=\{v_i\}$, $\mathrm{Pref}_i(u)=\{u_1\cdots u_i\}$ and $\mathrm{Pref}_i(v)=\{v_1\cdots v_i\}$ for every $i\leq\mathrm{min}\{n,m\}$. Hence, $r_2(u,v)=r_3(u,v)$ and therefore $d_2(u,v)=d_3(u,v)$.
\end{proof}

For the remaining of the section, we fix a trace monoid $M=\langle S\mid ab=ba; m_{ab}\neq\infty\rangle^+$, and set $p=|S|$. For every $u\in M$, let $\xi(u)$ denote the {\em support} of $u$, i.e. the set of atoms (elements of $S$) occurring in any expression of $u$.


\begin{lm}\label{lm10}
Let $u,v\in M$ such that $v\preceq u$. Set $\iota_2(u)=(u_1,\dots,u_n)$ and $\iota_2(v)=(v_1,\dots,v_k)$. Then $k\leq n$, and for all $i\leq k$, we have $$v_1\cdots v_i\preceq u_1\cdots u_i.$$
\end{lm}

\begin{proof}
Let $w$ be in $M$ and $s$ be in $S$. Assume $\iota_2(w)=(w_1,\dots,w_l)$, and set $\iota_2(ws)=(w'_1,\dots,w'_{l'})$. By \cite{VWy}, $\iota_2(ws)$ can be obtained in the following way. If $s\notin\xi(w)$ and $sw=ws$, then $\iota_2(ws)=(w_1s,\dots,w_l)$ and $|ws|_2=l$. If $s\in\xi(w_l)$ or $sw_l\neq w_ls$, then $\iota_2(ws)=(w_1,\dots,w_l,s)$ and $|ws|_2=l+1$. Otherwise, set $j_0=\min\{j\in\{1,\dots,l\};s\notin\xi(w_j\cdots w_l)~\text{and}~sw_j\cdots w_l=w_j\cdots w_ls\}$. We have $j_0<l$, $\iota_2(ws)=(w_1,\dots,w_{j_0}s,\dots,w_l)$ and $|ws|_2=l$. In all cases, $l'\geq l$ and $w_1\cdots w_i\preceq w'_1\cdots w'_i$ for all $i\leq l$.\\
Now we can write $u=vs_1\cdots s_m$ with $s_1,\dots,s_m$ in $S$, and apply the above argument to all the pairs $(vs_1\cdots s_{i-1},vs_1\cdots s_i)$ to conclude.
\end{proof}


\begin{lm}\label{lm11}
We have $$d_3\leq d_2.$$
\end{lm}

\begin{proof}
Let $u,v$ be distinct in $M$. Set $\iota_2(u)=(u_1,\dots,u_n)$ and $\iota_2(v)=(v_1,\dots,v_m)$. If $r_2(u,v)=0$, then $r_3(u,v)\geq r_2(u,v)$ and so $d_3(u,v)\leq d_2(u,v)$. Otherwise, let $w$ be in $M$ such that $\lS(w)=k=r_2(u,v)$. By Lemma \ref{lm10}, we have $w\preceq u~\Leftrightarrow~ w\preceq u_1\cdots u_k~\Leftrightarrow~ w\preceq v_1\cdots v_k~\Leftrightarrow~ w\preceq v$, so $\mathrm {Pref}_k(u)=\mathrm{Pref}_k(v)$. Thus, $r_3(u,v)\geq r_2(u,v)$ and so $d_3(u,v)\leq d_2(u,v)$.
\end{proof}~

\begin{lm}\label{lm12}
We have $$d_2^p\leq2^pd_3.$$
\end{lm}

\begin{proof}
Let $u,v$ be in $M$. Set $\iota_2(u)=(u_1,\dots,u_n)$ and $\iota_2(v)=(v_1,\dots,v_m)$, and denote $k=r_2(u,v)$. Since we have $m_{st}\in\{2,\infty\}$ for all $s,t\in S$, then $u_i=\Delta\{s\in S\mid s\preceq u_i\cdots u_n\}=\prod\{s\in S\mid s\preceq u_i\cdots u_n\}$ and $\lS(u_i)\leq p$ for all $i$. Assuming that $u\neq v$ and $n>k$, we have $u_1\cdots u_{k+1}\nprec v$. So $r_3(u,v)\leq\lS(u_1\cdots u_{k+1})\leq p(k+1)=pr_2(u,v)+p$. Thus, $-pr_2(u,v)\leq p-r_3(u,v)$. And therefore, $(d_2(u,v))^p\leq2^pd_3(u,v)$.
\end{proof}~






We may now prove our proposition.

\begin{proof}[Proof of Proposition \ref{prp5}]~
\begin{enumerate}
\item Gathering Lemmas \ref{lm11} and \ref{lm12}, we get immediately that $d_2$ and $d_3$ are uniformly equivalent.
\item Example \ref{ex5} i.
\item Assume $M\simeq\N^p$ is the free abelian monoid, i.e. $m_{ab}=2$ for all $a,b\in S$. Let $u,v$ be in $M$. Set $\iota_1(u)=(u_1,\dots,u_n)$ and $\iota_1(v)=(v_1,\dots,v_m)$. Since $ab=ba$ for all $a,b\in S$, we have $M_{\mathrm{red}}=\{\Delta(T)\mid T\subseteq S\}=X_2$. So, for every $w\in M$, $\iota_1(w)=\iota_2(w)$. Thus $r_1(u,v)=r_2(u,v)$ and  $d_1(u,v)=d_2(u,v)$. If $\xi(u)\neq\xi(v)$, then $r_1(u,v)=r_2(u,v)=r_3(u,v)=0$, and so $d_1(u,v)=d_2(u,v)=d_3(u,v)$. If $\xi(u)=\xi(v)=\{s_1,\dots,s_k\}$ and $u\neq v$, write $$u=s_1^{f_1}\cdots s_k^{f_k}~~~~\text{and}~~~~v=s_1^{g_1}\cdots s_k^{g_k}.$$ The set $\mathrm{T}=\{f_i,g_i\mid f_i\neq g_i\}$ is non empty because $u\neq v$. Let $q=\mathrm{min}\mathrm{(T)}$. We may suppose $q=f_j$ for some $1\leq j\leq k$. Thus, for all $i\leq q$, we have $u_i=v_i$, $s_j\nprec u_{q+1}$ and $s_j\preceq v_{q+1}$. Therefore $r_1(u,v)=r_2(u,v)=q$. On the one hand, $s_j^{q+1}\preceq v$ and $s_j^{q+1}\nprec u$, therefore $\mathrm{Pref}_{q+1}(u)\neq\mathrm{Pref}_{q+1}(v)$. On the other hand, let $w\in M$ such that $w\preceq u$ and $\lS(w)=q$. Then we can write $w=s_1^{h_1}\cdots s_k^{h_k}$ with $h_1+\cdots+h_k=q$. Since $w\preceq u$, for $1\leq i\leq k$, we have $h_i\leq f_i$. If $f_i=g_i$, then $h_i\leq g_i$. And if $f_i\neq g_i$, since $q=\mathrm{min}\mathrm{(T)}$ and $h_i\leq q$, one has $h_i\leq g_i$. So for all $1\leq i\leq k$, we have $h_i\leq g_i$, which means that $w\preceq v$. Similarly, if $w\preceq v$ with $\lS(w)=q$, then $w\preceq u$ as well. Hence, $\mathrm{Pref}_q(u)=\mathrm{Pref}_q(v)$, and therefore $r_3(u,v)=q$. Thus, $d_1(u,v)=d_2(u,v)=d_3(u,v)$.
\end{enumerate}
\end{proof}





\subsection{The case of Garside monoids}

In this subsection, we show that in a finitely generated Garside monoid, equiped with a length $\ell_S$, the metrics $d_1$ and $d_3$ are uniformly equivalent. Recall that a {\em Garside monoid} is a preGarside monoid containing a {\em Garside element}, i.e. a balanced element whose set of divisors generates the whole monoid. let $M$ be a Garside monoid with a Garside element $\Delta$. One of the Garside element important properties is that for all $u\in M$, $\alpha(u)$ is the greatest common (left) divisor of $u$ and $\Delta$, denoted by $u\wedge\Delta$. In other words, we have $u\in M_{\mathrm{red}}~\Leftrightarrow~u\preceq\Delta$.






\begin{prp}\label{prp6}
Let $M$ be a finitely generated Garside monoid, equiped with a length $\ell_S$. Set $\ell=\lS(\Delta)$. We have $$d_1^\ell\leq 2^\ell d_3.$$
\end{prp}

\begin{proof}
Let $u,v$ be distinct in $M$. Set $\iota_1(u)=(u_1,\dots,u_n)$, $\iota_1(v)=(v_1,\dots,v_m)$, $r_1=r_1(u,v)$, and $r_3=r_3(u,v)$. If $u=1$ or $v=1$, then $r_1=r_3=0$, so $(d_1(u,v))^\ell\leq2^\ell d_3(u,v)$. If $r_1=n$ or $r_1=m$, then $u\preceq v$ or $v\preceq u$, and so $r_3\leq r_1\ell$ because $\ell_S(w)\leq\ell$ for all $w\in M_{\mathrm{red}}$. Thus $(d_1(u,v))^\ell\leq d_3(u,v)\leq2^\ell d_3(u,v)$. Otherwise, suppose $u_1\cdots u_{r_1+1}\preceq v$ and $v_1\cdots v_{r_1+1}\preceq u$. Then $u_{r_1+1}\preceq v_{r_1+1}\cdots v_m$ and $v_{r_1+1}\preceq u_{r_1+1}\cdots u_n$. Thus, by definition of the greedy normal form, $u_{r_1+1}\preceq v_{r_1+1}$ and $v_{r_1+1}\preceq u_{r_1+1}$. So $u_{r_1+1}=v_{r_1+1}$, which contradicts the definition of $r_1(u,v)$. Therefore, one has either $u_1\cdots u_{r_1+1}\nprec v$ or $v_1\cdots v_{r_1+1}\nprec u$. Either way, we have $r_3\leq\mathrm{max}\{\lS(u_1\cdots u_{r_1+1}),\lS(v_1\cdots v_{r_1+1})\}\leq(r_1+1)\ell$. Hence, $(d_1(u,v))^\ell\leq2^\ell d_3(u,v)$.
\end{proof}

Gathering propositions \ref{prp3} and \ref{prp6}, we get :

\begin{thm}\label{thm3}
In a  finitely generated Garside monoid, equiped with a length $\ell_S$, the metrics $d_1$ and $d_3$ are uniformly equivalent.
\end{thm}



\underline{\bf Question :} Are $d_1$ and $d_2$ uniformly equivalent in Garside monoids ?

~~\\

\section{Contractability of endomorphisms of Artin monoids}

The aim of this section is to extend \cite[Theorem 4.1]{RoS} to all Artin monoids. However, it is easy to verify  that the assertions stated in \cite[Theorem 4.1]{RoS} can be not equivalent in an Artin monoids (see Example~\ref{ex6} below). So \cite[Theorem 4.1]{RoS} can not be directly extended. Actually, in the general case, the metric $d_1$ appears as more natural than $d_2$, mainly because of Property~(\ref{eq8}). Moreover, one can verify that  Property (iii) of Lemma \ref{lm13} is the exact translation of Property (14) in \cite[Theorem 4.1(iii)]{RoS} when replacing $d_2$ by $d_1$. So Theorem \ref{thm4} provided a convenient generalisation of  \cite[Theorem 4.1]{RoS} to the context of Artin monoids.


Let $M=\langle S\mid[a,b\rangle^{m_{ab}}=[b,a\rangle^{m_{ab}}; m_{ab}\neq\infty\rangle^+$ be an Artin monoid, and $\varphi$ be in $\mathrm{End}(M)$. As shown in \cite{RoS}, the metric space $(M,d_1)$ admits a completion $(\widehat{M},d_1)$ defined as follows. Let $\partial M$ consist of all infinite sequences of the form $u_1u_2
\cdots$, such that $u_i\in M_{\mathrm{red}}$ for all $i$, and $u_1\cdots u_n$ is a (greedy) normal 
form for all $n\in\N$. We have $\widehat{M}=M\cup\partial M$.

The metric $d_1$ extends to $\widehat{M}$ in the obvious way, and it is easy to check that $(\widehat{M},d_1)$ is
complete: given a Cauchy sequence $(U_n)_n$ with $U_n=u_{n1}u_{n2}\cdots$, it follows
easily that each sequence $(u_{nk})_k$ is stationary with limit, say, $u_k$, and we get $u_1u_2\cdots
=\mathrm{lim}_{n\rightarrow\infty}U_n$. Since $u_1\cdots u_n\in M$ and it is in a normal form for all
$n$, and $u_1u_2\cdots=\mathrm{lim}_{n\rightarrow\infty}u_1\cdots u_n$, then $(\widehat{M},d_1)$ is
indeed the completion of $(M,d_1)$. We may refer to $\partial M$ as the {\em boundary} of $M$.

Assume that $\varphi$ is uniformly continuous with respect to $d_1$. Since
$(\widehat{M},d_1)$ is the completion of $(M,d_1)$, $\varphi$ admits a unique continuous extension $\Phi$
to $(\widehat{M},d_1)$. By continuity, we must have $\Phi(X)=\mathrm{lim}_{n\rightarrow\infty}\varphi(u_n)$
whenever $X\in\partial M$ and $(u_n)_n$ is a sequence on $M$ satisfying $X=\mathrm{lim}_{n\rightarrow
\infty}u_n$.~~\\

\begin{lm}\label{lm13}
The following properties are equivalent:

{\renewcommand{\arraystretch}{1.5}
\renewcommand{\tabcolsep}{0.2cm}
\begin{tabular}{lllcl}
(i) & for all $u,v\in M$, & $\alpha(uv)=\alpha(u)$ & $\Rightarrow$ & $\alpha(\varphi(uv))=\alpha(\varphi(u));$\\
(ii) & for all $u,v\in M$, & $\alpha(uv)=u$ & $\Rightarrow$ & $\alpha(\varphi(uv))=\alpha(\varphi(u))$;\\
(iii) & for all $u,v\in M_{\mathrm{red}}$, & $\alpha(uv)=u$ & $\Rightarrow$ & $\alpha(\varphi(uv))=\alpha(\varphi(u))$;\\
(iv) & for all $u\in M$, & $\alpha(\varphi(u))=\alpha(\varphi(\alpha(u)))$.
\end{tabular}}
\end{lm}~

\begin{proof}
We prove that $(ii)\Rightarrow(iv)\Rightarrow(i)\Rightarrow(iii)\Rightarrow(ii)$. Let $u,v$ be in $M$. Set $\iota_1(u)=(u_1,\dots,u_n)$ the greedy normal form of $u$. Assume $(ii)$ holds. Then $\alpha(u_1(u_2\cdots u_n))=u_1$ and, by $(ii)$, $\alpha(\varphi(u_1\cdots u_n))=\alpha(\varphi(u_1))=\alpha(\varphi(\alpha(u)))$. Thus $(iv)$ holds.\\
Assume $(iv)$ holds and $\alpha(uv)=\alpha(u)$. Then by $(iv)$, $\alpha(\varphi(uv))=\alpha(\varphi(\alpha(uv)))=\alpha(\varphi(\alpha(u)))= \alpha(\varphi(u))$ and $(i)$ holds.\\
Assume now $(i)$. If $u\in M_{\mathrm{red}}$ then $\alpha(u)=u$, so $\alpha(uv)=u$ implies $\alpha(uv)=\alpha(u)$, which in turn implies $\alpha(\varphi(uv))=\alpha(\varphi(u))$ by $(i)$. So $(iii)$ holds.\\
Assume finally $(iii)$ and assume $\alpha(uv)=u$. In particular $u\in M_{\mathrm{red}}$. We prove that $\alpha(\varphi(uv))=\alpha(\varphi(u))$ by induction on $|v|_1=k$. If $|v|_1=1$, then $v\in M_{\mathrm{red}}$ and the result holds by (iii). Assume $|v|_1>1$ plus the induction hypothesis. Set $k=|v|_1$ and let $\iota_1(v)=(v_1,\dots,v_m)$ be the greedy normal form of $v$. In view of (\ref{eq8}), we have $\alpha(\varphi(uv))=\alpha(\varphi
(u)\alpha(\varphi(v)))$. Now, we have $\alpha(v)=\alpha(v_1
\cdots v_k)=v_1$. Since $|v_2\cdots v_k|_1=k-1<k$, then by the induction hypothesis, $\alpha(\varphi
(v_1\cdots v_k))=\alpha(\varphi(v_1))$. Hence $\alpha(\varphi(uv))=\alpha(\varphi(u)\alpha(\varphi(v
_1)))=\alpha(\varphi(u)\varphi(v_1))=\alpha(\varphi(uv_1))$. We also have $\alpha(uv_1)=\alpha(u
\alpha(v))=\alpha(uv)=u$ and $|v_1|_1=1<k$. Then, by the case $k=1$, we get $\alpha(\varphi
(uv_1))=\alpha(\varphi(u))$. Thus $\alpha(\varphi(uv))=\alpha(\varphi(u))$ and $(ii)$ holds.
\end{proof}~

Recall that a mapping $\varphi:(X,d)\rightarrow(X,d)$ on a metric space is called a {\em contraction} with respect to $d$, if $d(\varphi(u),\varphi(v))\leq d(u,v)$ for all $u,v\in X$.

~~\\
~~\\

\begin{thm}\label{thm4}
The following properties are equivalent:
\begin{enumerate}
\item[(i)] $\varphi$ is uniformly continuous, and $\Phi$ is a contraction with respect to $d_1$;
\item[(ii)] $\varphi$ is a contraction with respect to $d_1$;
\item [(iii)] for all $u,v\in M_{\mathrm{red}}$, ~~~$\alpha(uv)=u~~\Rightarrow~~\alpha(\varphi(uv))=\alpha(\varphi(u))$;
\item [(iv)] for all $u\in M$, ~~$\alpha(\varphi(u))=\alpha(\varphi(\alpha(u)))$.
\end{enumerate}
Furthermore, in these cases, if $u=u_1u_2\cdots\in\widehat{M}$ and $\Phi(u)=U_1U_2\cdots$, then for all $m\in\N^*$ with $m\leq|u|_1$, one has~~$\iota_1(\varphi(u_1\cdots u_m))=(U_1,\dots,U_m,\dots)$.
\end{thm}

\begin{proof}
The equivalence $(i)\Leftrightarrow(ii)$ is clear, and $(iii)\Leftrightarrow(iv)$ follows from Lemma \ref{lm13}.\\
Assume $(ii)$. Let $u,v$ be in $M_{\mathrm{red}}$ such that $\alpha(uv)=u$. Then $d_1(uv,u)=\frac{1}{2}$, and by $(ii)$, $d_1(\varphi(uv),\varphi(u))\leq\frac{1}{2}$. Thus $\alpha(\varphi(uv))=\alpha(\varphi(u))$. So $(iii)$ holds.\\
Conversely, assume $(iii)$. Let $u$ belong to $M$. Set $\iota_1(u)=(u_1,\dots,u_n)$ and $\iota_1(\varphi(u))=(U_1,\dots,U_N)$. We prove by induction on $k$ that for $k\in\{1,\dots,n\}$, one has $\iota_1(\varphi(u_1\cdots u_k))=(U_1,\dots,U_k,\dots)$. In particular $n\leq N$. For $k=1$, the result holds by $(iv)$. So assume $k\geq 2$ plus the induction hypothesis. By the induction hypothesis, we can write $\varphi(u_1\cdots u_{k-1})=U_1\cdots U_{k-1}Z$, with $Z\in M$. Since $\varphi(u_1\cdots u_n)=U_1\cdots U_N$, it follows by cancellativity that $Z\varphi(u_k\cdots u_n)=U_k\cdots U_N$. Therefore, $\alpha(Z\varphi(u_k\cdots u_n))=U_k$. But, using (\ref{eq8}) and $(iv)$, we have $\alpha(Z\varphi(u_k\cdots u_n))=\alpha(Z\alpha(\varphi(u_k\cdots u_n)))=\alpha(Z\alpha(\varphi(u_k)))$. In particular, $U_k$ left divides $Z\alpha(\varphi(u_k))$. Hence, $U_1\cdots U_k$ left divides $\varphi(u_1\cdots u_k)$. This imposes by definition of the greedy normal form that $\iota_1(\varphi(u_1\cdots u_k))=(U_1,\dots,U_k,\dots)$, which proves the induction step. Now let $v$ belong to $M$. Set $\iota_1(v)=(v_1,\dots,v_p)$ and $\iota_1(\varphi(v))=(V_1,\dots,V_P)$. Assume $d_1(u,v)=2^{-k}$. Then $u_1=v_1,\dots,u_k=v_k$. It follows from the above result that $\iota_1(\varphi(u_1\cdots u_k))=(U_1,\dots,U_k,\dots)=(V_1,\dots,V_k,\dots)$. Thus $U_1=V_1,\dots,U_k=V_k$ and $d_1(\varphi(u),\varphi(v))\leq2^{-k}$. So $\varphi$ is a contraction. Hence, $(ii)$ holds.\\
Finally, assume $(i)$. Let $u$ lie in $\widehat{M}$. Say $u=u_1u_2\cdots$. Set $\Phi(u)=U_1U_2\cdots\in\widehat{M}$. Let $m$ be in $\N^*$. Then $d_1(u_1\cdots u_m,u)=2^{-m}$. So by $(i)$, we have $d_1(\Phi(u_1\cdots u_m),\Phi(u))\leq2^{-m}$. Hence, $\iota_1(\varphi(u_1\cdots u_m))=\iota_1(\Phi(u_1\cdots u_m))=(U_1,\dots,U_m,\dots)$.
\end{proof}~~

The following example illustrates that the equivalence between (ii) and (iii) in \cite[Theorem 4.1]{RoS} is not true for all Artin monoïds with respect to $d_2$, which is why we used $d_1$ to extend \cite[Theorem 4.1]{RoS} into our Theorem \ref{thm4}.

\begin{ex}\label{ex6}
Let $M=\langle s,t\mid ststststs=tstststst\rangle^+$, and $\varphi\in\mathrm{End}(M)$ such that $\varphi(s)=sts$ and $\varphi(t)=tst$. Set $\Delta=ststststs$, and define $\alpha_2(u)$ by $\iota_2^{[1]}(u)=(\alpha_2(u))$ for $u$ in $M$. We have $X_2=\{s,t,\Delta\}$, and $\{(u,v)\in X_2^2\mid\alpha_2(uv)=u\}=\{(s,t),(t,s),(s,s),(t,t),(\Delta,s),(\Delta,t),(\Delta,\Delta)\}$. Then, for all $u,v\in X_2$, we have $\alpha_2(uv)=u~\Rightarrow~\alpha_2(\varphi(uv))=\alpha_2(\varphi(u))$. However, the morphism $\varphi$ is not a contraction with respect to $d_2$, since $d_2(\varphi(s),\varphi(sts))>d_2(s,sts)$.
\end{ex}~


\begin{thebibliography}{10}

\bibitem{BDM}
{\sc Bessis, D., Digne, F., and Michel, J.}
\newblock Springer theory in braids groups and the {B}irman-{K}o-{L}ee monoid.
\newblock {\em Pacific J. Math. 205\/} (2002), 287--309.

\bibitem{BMP}
{\sc Bonizzoni, P., Mauri, G. and Pighizzini, G.}
\newblock About infinite traces, in: {V}. {D}iekert (ed.), {P}roceding of the {ASMICS} {W}orshops on {P}artially {C}ommutative {M}onoids.
\newblock {\em Tech. Rep. TUM-I 9002\/} (1990), 1--10.


\bibitem{BrS}
{\sc Brieskorn, E., and Saito, K.}
\newblock Artin {G}ruppen und {C}oxeter {G}ruppen.
\newblock {\em Invent. Math. 17\/} (1972), 245--271.

\bibitem{Cri2}
{\sc Crisp, J.}
\newblock Symmetrical subgroups of {A}rtin groups.
\newblock {\em Adv. in Math. 152\/} (2000), 159--177.

\bibitem{DeP}
{\sc Dehornoy, P., and Paris, L.}
\newblock {G}aussian groups and {G}arside groups, two generalisations of
  {A}rtin groups.
\newblock {\em Proc. London Math Soc. 79}, 3 (1999), 569--604.

\bibitem{GoP}
{\sc Godelle, E., and Paris, L.}
\newblock Pregarside monoids and groups, parabolicity, amalgamation, and {FC}
  property.
\newblock {\em I.J.A.C 23\/} (2013), 1431--1467.

\bibitem{KK}
{\sc Kummetz, R., and Kuske, D.}
\newblock The topology of mazurkiewicz traces.
\newblock {\em Theoret. Comp. Sci. 305\/} (2003), 237--258.

\bibitem{Kwi}
{\sc Kwiatkowska, M., Z.}
\newblock A metric for traces.
\newblock {\em Information processing Letters 35\/} (1990), 129--135.

\bibitem{Mich}
{\sc Michel, J.}
\newblock A note on braid monoids.
\newblock {\em J. of Algebra 215\/} (1999), 366--377.

\bibitem{RoS}
{\sc Rodaro, E., and Silva, P.~V.}
\newblock Fixed points of endomorphisms of trace monoids.
\newblock {\em arXiv:1211.4517v1\/}.

\bibitem{RoS1}
{\sc Rodaro, E., and Silva, P.~V.}
\newblock Fixed points of endomorphisms of trace monoids.
\newblock {\em Semigroup {F}orum\/} November (2013), DOI 10.1007/s00233-013-9553-0.

\bibitem{VWy}
{\sc {Van Wyk}, L.}
\newblock Graph groups are biautomatic.
\newblock {\em J. Pure Appl. Algebra 94\/} (1994), 341--352.

\end{thebibliography}

\end{document}